\documentclass[12pt]{amsart}
\usepackage{amsmath,amsthm,amsfonts,amscd,amssymb,eucal,latexsym,mathrsfs, appendix, yhmath, setspace}
\usepackage[numbers,sort&compress]{natbib}
\usepackage[all,cmtip]{xy}
\usepackage{enumerate}

\newtheorem{theorem}{Theorem}[section]
\newtheorem{corollary}[theorem]{Corollary}
\newtheorem{lemma}[theorem]{Lemma}
\newtheorem{proposition}[theorem]{Proposition}

\theoremstyle{definition}
\newtheorem{definition}[theorem]{Definition}

\newtheorem{notation}[theorem]{Notation}

\newcommand{\cA}{{\mathcal A}}

\newcommand{\cF}{{\mathcal F}}

\newcommand{\cM}{{\mathcal M}}

\newcommand{\cP}{{\mathcal P}}
\newcommand{\cQ}{{\mathcal Q}}

\newcommand{\cU}{{\mathcal U}}
\newcommand{\cV}{{\mathcal V}}

\newcommand{\Fb}{{\mathbb F}}

\newcommand{\Nb}{{\mathbb N}}

\newcommand{\Rb}{{\mathbb R}}

\newcommand{\Zb}{{\mathbb Z}}
\newcommand{\sA}{{\mathscr A}}

\newcommand{\sW}{{\mathscr W}}

\newcommand{\diam}{{\rm diam}}

\newcommand{\hnv}{{h^{\rm nv}}}

\allowdisplaybreaks

\begin{document}

\title{Combinatorial Independence and Naive Entropy}

\author{Hanfeng Li}
\author{Zhen Rong}

\address{\hskip-\parindent
H.L., Center of Mathematics, Chongqing University,
Chongqing 401331, China.\\
Department of Mathematics, SUNY at Buffalo,
Buffalo, NY 14260-2900, USA.}
\email{hfli@math.buffalo.edu}

\address{\hskip-\parindent
Z.R., College of Statistics and Mathematics, Inner Mongolia University of Finance and Economics, Hohhot 010000, China.}
\email{rongzhen@cqu.edu.cn}

\date{April 10, 2020}


\begin{abstract}
We study the independence density for finite families of finite tuples of sets for continuous actions of discrete groups on compact metrizable spaces.
We use it to show that actions with positive naive entropy are Li-Yorke chaotic and untame. In particular, distal actions have zero naive entropy. This answers a question of Lewis Bowen.
\end{abstract}

\maketitle

\section{Introduction} \label{S-introduction}

Let a countably infinite group $\Gamma$ act on a compact metrizable space $X$ continuously.
Motivated by the consideration in \cite{Bowen17} for the naive entropy of measure-preserving actions, Burton introduced
the {\it naive topological entropy} of $\Gamma\curvearrowright X$ in
\cite{Burton}. This is also studied in \cite{DFR}. For a finite open cover $\cU$ of $X$, denote by $N(\cU)$ the minimal cardinality of subcovers of $\cU$. For any nonempty finite subset $F$ of $\Gamma$, set $\cU^F=\bigvee_{s\in F}s^{-1}\cU$. The naive entropy of $\cU$ is defined as
$$ \hnv(\Gamma, \cU):=\inf_{F}\frac{1}{|F|}\log N(\cU^F),$$
where $F$ ranges over nonempty finite subsets of $\Gamma$.
The naive entropy of $\Gamma\curvearrowright X$ is defined as
$$ \hnv(\Gamma \curvearrowright X):=\sup_\cU\hnv(\Gamma, \cU)$$
for $\cU$ ranging over finite open covers of $X$.

It is known that the naive entropy $\hnv(\Gamma\curvearrowright X)$ coincides with the classical topological entropy when $\Gamma$ is amenable \cite[Theorem 6.8]{DFR}.
When $\Gamma$ is sofic, if $\hnv(\Gamma\curvearrowright X)=0$, then the sofic topological entropy of $\Gamma\curvearrowright X$ with respect to any sofic approximation sequence of $\Gamma$ is either $-\infty$ or $0$ \cite[Theorem 1.1]{Burton} \cite[Propositions 4.6 and 4.16]{KL13}.
When $\Gamma$ is nonamenable, $\hnv(\Gamma\curvearrowright X)$ is either $0$ or $\infty$ \cite[Section 2.2]{Burton}. Thus for nonamenable $\Gamma$, the naive entropy just describes the action $\Gamma\curvearrowright X$ as having positive entropy or zero entropy.

Initiated by the work of Blanchard \cite{Blanchard92, Blanchard93}, the local entropy theory developed quickly \cite{BGH, BGKM, CL, Glanser03, GY, HLY, HMRY, HY06, HY09, HYZ, KL07, KL09, KL13, KL16}.  A combinatorial approach was given to this theory in \cite{KL07}. It turns out that the combinatorial approach enables us to give a unified treatment for several dynamical properties. In general, one considers tuples of subsets of $X$ which have large independence sets in $\Gamma$ (see Definition~\ref{D-indep set} below), and then localizes to tuples of points in $X$ for which the tuple of subsets associated to any product neighborhood has large independence sets. Different largeness then corresponds to different dynamical properties. For instance, positive density corresponds to positive entropy for actions of amenable group \cite{HY06, KL07, KL16}, infinite sets corresponds to untameness \cite{KL07, KL16}, and arbitrary large finite sets corresponds to nonnullness \cite{KL07}. The correspondence between positive density and positive entropy also holds for actions of sofic groups \cite{KL13, KL16}, though the density is defined using the sofic approximation sequence instead.

A natural question is whether positive naive entropy can be studied using combinatorial independence. Indeed a notion of density was introduced for tuples of subsets for actions of any group in \cite[Definition 3.1]{KL13}, and the corresponding type of tuples of points in $X$ was also introduced in \cite[Definition 3.2]{KL13}. However, in general it is impossible to localize positive density from tuples of subsets to a tuple of points (see Proposition~\ref{P-positive naive no orbit IE}). The novelty in this paper is that we shall stay at the level of tuples of subsets and consider finite families of tuples of subsets instead of a single tuple (see Definition~\ref{D-indep density}). It turns out that this characterizes positive naive entropy (Theorem~\ref{T-positive naive}), and we can use it to obtain some interesting properties of actions with positive naive entropy.

The action $\Gamma\curvearrowright X$ is said to be {\it Li-Yorke chaotic} \cite{BGKM, LY} if there is an uncountable set $Y\subseteq X$ such that for any distinct $x, y\in Y$, one has
$$ \limsup_{\Gamma\ni s\to \infty}\rho(sx, sy)>0 \mbox{ and } \liminf_{\Gamma\ni s\to \infty}\rho(sx, sy)=0,$$
where $\rho$ is any given compatible metric on $X$. Using measure-dynamical techniques Blanchard et al.  showed first that positive entropy implies Li-Yorke chaos for continuous maps \cite{BGKM}. This was extended to actions of amenable groups \cite[Corollary 3.19]{KL07} and sofic groups \cite[Corollary 8.4]{KL13} using combinatorial independence. Here using independence density for finite families of tuples of subsets we extend this implication to actions of all groups.

\begin{theorem} \label{T-Li Yorke}
For any countably infinite group $\Gamma$, any continuous action of $\Gamma$ on a compact metrizable space with positive naive entropy is Li-Yorke chaotic.
\end{theorem}

For sofic groups, in fact Theorem~\ref{T-Li Yorke} is stronger than  \cite[Corollary 8.4]{KL13} since there are actions with zero sofic entropy but positive naive entropy (see the discussion at the end of Section~\ref{S-example}).

For any $\Gamma$-invariant Borel probability measure $\mu$ on $X$, Bowen introduced the {\it naive entropy} \cite[Definition 7]{Bowen17} \cite[Definition 2.2]{Burton} of the measure-preserving action $\Gamma\curvearrowright (X, \mu)$ by
$$h^{\rm nv}_\mu(\Gamma\curvearrowright X):=\sup_{\cP}\inf_F\frac{1}{|F|}H_\mu(\cP^F),$$
where $\cP$ ranges over finite Borel partitions of $X$ and $F$ ranges over nonempty finite subsets of $\Gamma$. Here for a finite Borel partition $\cQ$ of $X$, $H_\mu(\cQ)$ denotes
the Shannon entropy $-\sum_{q\in \cQ}\mu(q)\log \mu(q)$. It is easy to check that when $\Gamma$ is amenable, $h^{\rm nv}_\mu(\Gamma\curvearrowright X)$ coincides with the classical Kolmogorov-Sinai entropy \cite[Theorem 4.2]{DFR} \cite[page 198]{KL16}. When $\Gamma$ is nonamenable, $h^{\rm nv}_\mu(\Gamma\curvearrowright X)$  is either $0$ or $+\infty$ \cite[Theorem 2.13]{Bowen17}. Burton showed that one always has $h^{\rm nv}_\mu(\Gamma\curvearrowright X)\le h^{\rm nv}(\Gamma\curvearrowright X)$ \cite[Theorem 1.3]{Burton}.

The action $\Gamma\curvearrowright X$ is called {\it distal} \cite{Auslander} if for any distinct $x,y\in X$ one has $\inf_{s\in \Gamma}\rho(sx, sy)>0$, where $\rho$ is any given compatible metric on $X$. Parry showed first that distal actions of $\Zb$ have zero entropy \cite{Parry}. Since distal actions cannot be Li-Yorke chaotic, it was observed in \cite[Corollary 8.5]{KL13} that distal actions of sofic groups have sofic entropy either $-\infty$ or $0$. Via reduction to actions of $\Zb$, Burton \cite[Example 2.2]{Burton} showed that if $\Gamma$ contains an element with infinite order, then any distal action of $\Gamma$ has zero naive entropy. From Theorem~\ref{T-Li Yorke} and the above paragraph we conclude that this holds for all groups, which answers a question of Bowen \cite[Question 8]{Bowen17}.

\begin{corollary} \label{C-distal}
For any countably infinite group $\Gamma$, any distal continuous action of $\Gamma$ on a compact metrizable space $X$ has zero naive topological entropy. If $\mu$ is a $\Gamma$-invariant Borel probability measure on $X$, then the action $\Gamma\curvearrowright (X, \mu)$ also has zero naive entropy.
\end{corollary}

The notion of tame actions was introduced by K\"{o}hler \cite{Kohler} motivated by Rosenthal's characterization of Banach spaces containing $\ell^1$ \cite{Rosenthal}, and is well studied \cite{Aujogue, CS, Glasner06, Glasner07, Glasner18, GM06, GM12, GM14, GM18a, GM18, GM19, GMU, Huang, Ibarlucia, KL07, KL16, LR, Romanov}. Denote by $C(X)$ the space of all continuous $\Rb$-valued functions on $X$ equipped with the supremum norm. The action $\Gamma\curvearrowright X$ is said to be {\it untame} if there are some $f\in C(X)$ and some infinite subset $H$ of $\Gamma$ such that the map $\delta_s\mapsto sf$ for $s\in H$ extends to a linear Banach space isomorphism from $\ell^1(H)$ to the closed linear span of $sf$ for $s\in H$ in $C(X)$. Tameness can also be characterised in terms of the Ellis semigroup of $\Gamma\curvearrowright X$. Using combinatorial independence it was shown that positive entropy actions are untame for amenable groups \cite{KL07} and sofic groups \cite{KL13}. Here we extend it to all groups in the context of naive entropy.

\begin{theorem} \label{T-naive to tame}
For any countably infinite group $\Gamma$, tame continuous actions on compact metrizable spaces have zero naive entropy.
\end{theorem}

This paper is organized as follows. We introduce the independence density for finite families of subsets in Section~\ref{S-independence density}, and show that positive independence density characterizes positive naive entropy. Theorems~\ref{T-Li Yorke} and \ref{T-naive to tame} are proved in Sections~\ref{S-chaos} and \ref{S-tame} respectively.
In Section~\ref{S-example} we exhibit an action with positive naive entropy but no non-diagonal orbit IE-pairs. This example shows that in general one cannot localize positive density from tuples of subsets to a tuple of points.

Throughout this article, $\Gamma$ will be a countably infinite discrete group with identity element $e_\Gamma$, and we fix a continuous action of $\Gamma$ on a compact metrizable space $X$. For any set $H$ we denote by $\cF(H)$ the set of nonempty finite subsets of $H$.  For each $n\in \Nb$ we write $[n]$ for $\{1, \dots, n\}$.

\noindent{\it Acknowledgments.}
H. L. is partially supported by NSF and NSFC grants. We are grateful to Lewis Bowen for comments.

\section{Independence Density for Families of Tuples} \label{S-independence density}

In this section we introduce the independence density for finite families of subsets and prove Theorem~\ref{T-positive naive}.

For each $k\in \Nb$, denote by $\sA_k$ the space of all $k$-tuples of subsets of $X$. Set $\sA=\bigcup_{k\in \Nb}\sA_k$, and $\sA_{\ge m}=\bigcup_{k\ge m}\sA_k$.

Recall the notion of independence sets introduced in \cite[Definition 2.1]{KL07} (see also \cite[Definition 8.7]{KL16}).

\begin{definition} \label{D-indep set}
For any $\mathbf A=(A_1, \dots, A_k)\in \sA$, we say $J\subseteq \Gamma$ is an {\it independence set} for $\mathbf A$ if for any nonempty finite set $F\subseteq J$ and any map $\omega: F\rightarrow [k]$ one has $\bigcap_{s\in F}s^{-1}A_{\omega(s)}\neq \emptyset$.
\end{definition}

\begin{definition} \label{D-indep density}
For any  finite $\cA\subseteq \sA$, we define the {\it independence density} of $\cA$ to be the largest $q\ge 0$ such that for every $F\in \cF(\Gamma)$ there are some $J\subseteq F$ with $|J|\ge q|F|$ and some ${\mathbf A}\in \cA$ so that $J$ is an independence set for ${\mathbf A}$.
\end{definition}

When $\cA$ consists of a single tuple, Definition~\ref{D-indep density} reduces to \cite[Definition 3.1]{KL13}.

We say that ${\mathbf A}=(A_1, \dots, A_k)\in \sA_k$ is {\it pairwise disjoint} ({\it closed} resp.) if the sets $A_1, \cdots, A_k$ are pairwise disjoint (closed resp.). We say that $\cA\subseteq \sA$ is {\it pairwise disjoint} ({\it closed} resp.) if each ${\mathbf A}\in \cA$ is pairwise disjoint (closed resp.).

For covers $\cU$ and $\cV$ of $X$, we denote by $\cU\vee \cV$ the cover of $X$ consisting of $U\cap V$ for $U\in \cU$ and $V\in \cV$.
We say that $\cU$ {\it is finer than} $\cV$ if every item of $\cU$ is contained in some item of $\cV$.
The following lemma is well known, see for example the proofs of \cite[Proposition 1]{Blanchard93} or \cite[Lemma 12.11]{KL16}.

\begin{lemma} \label{L-two cover}
For any finite open cover $\cU$ of $X$, there are $n\in \Nb$ and two-element open covers $\cU_1, \dots, \cU_n$ of $X$ such that $\bigvee_{j=1}^n\cU_j$ is finer than $\cU$.
\end{lemma}

Let $k\ge 2$ and let $Z$ be a nonempty finite set. Let $\sW$ be the cover of $\{0, 1, \dots, k\}^Z=\prod_{z\in Z}\{0, 1, \dots, k\}$ consisting of subsets of the form $\prod_{z\in Z}(\{0, 1, \dots, k\}\setminus \{i_z\})$, where $i_z\in [k]$ for each $z\in Z$. For a set $S\subseteq \{0, 1, \dots, k\}^Z$ we write $N_S$ for the minimal number of sets in $\sW$ needed to cover $S$. The following is the major combinatorial fact we need \cite[Lemma 3.3]{KL07} \cite[Lemma 12.13]{KL16}.

\begin{lemma} \label{L-comb1}
Let $k\ge 2$ and $b>0$. There exists $c>0$ depending only on $k$ and $b$ such that for any finite set $Z$ and $S\subseteq \{0,1, \dots, k\}^Z$ with $N_S\ge k^{b|Z|}$ there is a $J\subseteq Z$ with $|J|\ge c|Z|$ and $S|_J\supseteq [k]^J$.
\end{lemma}

The following theorem characterizes positive naive entropy in terms of finite pairwise disjoint closed families with positive independence density.

\begin{theorem} \label{T-positive naive}
The following are equivalent:
\begin{enumerate}
\item $\hnv(\Gamma\curvearrowright X)>0$,
\item  there is a finite pairwise disjoint closed $\cA\subseteq \sA_2$ with positive independence density,
\item there is a finite pairwise disjoint closed $\cA\subseteq \sA_{\ge 2}$ with positive independence density.
\end{enumerate}
\end{theorem}
\begin{proof} (1)$\Rightarrow$(2). Suppose that $\hnv(\Gamma\curvearrowright X)>0$. Then $\hnv(\Gamma, \cU)>0$ for some finite open cover $\cU$ of $X$. By Lemma~\ref{L-two cover} we can find two-element open covers $\cU_1, \dots, \cU_n$ of $X$ such that $\bigvee_{j=1}^n\cU_j$ is finer than $\cU$. We may assume that none of $\cU_j$ contains $X$. For each $1\le j\le n$, write $\cU_j$ as $\{U_{j, 1}, U_{j, 2}\}$ and
set ${\mathbf A}_j=(X\setminus U_{j, 1}, X\setminus U_{j, 2})\in \sA_2$. Then $\cA=\{{\mathbf A}_1, \dots, {\mathbf A}_n\}\subseteq \sA_2$ is finite pairwise disjoint and closed.
We claim that $\cA$ has positive independence density. Set $b:=\hnv(\Gamma, \cU)/(n\log 2)>0$. Then we have the constant $c>0$ in Lemma~\ref{L-comb1} depending only on $k=2$ and $b$. Let $F\in \cF(\Gamma)$. Then
$$ \hnv(\Gamma, \cU)|F|\le \log N(\cU^F)\le \log N(\bigvee_{j=1}^n\cU_j^F)\le \sum_{j=1}^n\log N(\cU_j^F).$$
Thus there is some $1\le j\le n$ with $\frac{\hnv(\Gamma, \cU)}{n}|F|\le \log N(\cU_j^F)$. Consider the map $\varphi:X\rightarrow \{0, 1, 2\}^F$ defined by
\begin{align*}
(\varphi(x))(s)=\left\{
    \begin{array}{ll}
    i, \mbox{ if } sx\in X\setminus U_{j, i} \mbox{ for some } i\in [2], \\
    0, \mbox{ otherwise}.
    \end{array}
    \right.
\end{align*}
Then $N_{\varphi(X)}=N(\cU_j^F)\ge 2^{b|F|}$. Therefore there is some $J\subseteq F$ with $|J|\ge c|F|$ and $\varphi(X)|_J\supseteq [2]^J$. Then $J$ is an independence set for ${\mathbf A}_j$.
Thus $\cA$ has independence density at least $c$.

(2)$\Rightarrow$(3) is trivial.

(3)$\Rightarrow$(1). Let $\cA=\{{\mathbf A}_1, \dots, {\mathbf A}_n\}\subseteq \sA_{\ge 2}$ be finite pairwise disjoint closed with independence density $q>0$.
For each $1\le j\le n$, write ${\mathbf A}_j$ as $(A_{j, 1}, \dots, A_{j, k_j})$ and set $V_j=X\setminus \bigcup_{i=1}^{k_j}A_{j, i}$ and $\cU_j=\{A_{j, 1}\cup V_j, \dots, A_{j, k_j}\cup V_j\}$. Then $\cU_1, \dots, \cU_n$ are finite open covers of $X$. Set $\cU=\bigvee_{j=1}^n\cU_j$. We claim that $\hnv(\Gamma, \cU)>0$. Let $F\in \cF(\Gamma)$. Then there are some $J\subseteq F$ with $|J|\ge q|F|$ and some $1\le j\le n$ such that $J$ is an independence set for ${\mathbf A}_j$. We have
$$ N(\cU^F)\ge N(\cU_j^F)\ge N(\cU_j^J)\ge k_j^{|J|},$$
and hence
$$\frac{1}{|F|}\log N(\cU^F)\ge \frac{|J|}{|F|}\log k_j\ge q\log k_j\ge q\log 2.$$
Therefore $\hnv(\Gamma, \cU)\ge q\log 2>0$.
\end{proof}

Let $\cA\subseteq \sA$. We say that $\cA'\subseteq \sA$  is a {\it simple splitting} of $\cA$ if there are some ${\mathbf A}\in \cA$ with ${\mathbf A}=(A_1, \dots, A_k)$
and some $1\le j\le k$ and $A_j=A_{j, 1}\cup A_{j, 2}$ such that
$$ \cA'=(\cA\setminus \{{\mathbf A}\})\cup \{(A_1, \dots, A_{j-1}, A_{j, 1}, A_{j+1}, \dots, A_k), (A_1, \dots, A_{j-1}, A_{j, 2}, A_{j+1}, \dots, A_k)\}. $$
We say that $\cA'\subseteq \sA$ is a {\it splitting} of $\cA$ if there are $\cA=\cA_1, \cA_2, \dots, \cA_m=\cA'$ such that $\cA_{j+1}$ is a simple splitting of $\cA_j$ for all $1\le j\le m-1$. Clearly splittings of pairwise disjoint families are still pairwise disjoint.

We need the following lemma \cite[Lemma 3.7]{KL07} \cite[Lemma 12.16]{KL16}, which is a consequence of Karpovsky and Milman's generalization of the Sauer-Perles-Shelah lemma \cite{KM, Sauer, Shelah}.

\begin{lemma} \label{L-comb2}
Let $k\ge 1$. Then there is some $c>0$ depending only on $k$ such that for any ${\mathbf A}\in \sA_k$, any simple splitting $\{{\mathbf A}_1, {\mathbf A}_2\}$ of $\{\mathbf A\}$, and any finite independence set $J$ for $\mathbf A$, there is an $I\subseteq J$ such that $|I|\ge c|J|$ and $I$ is an independence set for at least one of ${\mathbf A}_1$ and ${\mathbf A}_2$.
\end{lemma}

From Lemma~\ref{L-comb2}  we see that for any finite $\cA\subseteq \sA$ with positive independence density, every simple splitting of $\cA$ has positive independence density. Via induction we get

\begin{proposition} \label{P-splitting}
Let $\cA\subseteq \sA$ be finite with positive independence density. Then  every splitting of $\cA$ has positive independence density.
\end{proposition}

\section{Positive Independence Density and Li-Yorke Chaos} \label{S-chaos}

In this section we prove Theorem~\ref{T-chaos}, which shows that positive independence density implies Li-Yorke chaos.

\begin{notation} \label{N-power}
Let $E\in \cF(\Gamma)$. For ${\mathbf A}=(A_1, \dots, A_k)\in \sA_k$, we write ${\mathbf A}^E$ for the tuple in $\sA_{k^{|E|}}$ consisting of $\bigcap_{s\in E}s^{-1}A_{\omega(s)}$ for all $\omega\in [k]^E$ in any order.
\end{notation}

For $K, E\in \cF(\Gamma)$, we say that $E$ is {\it $K$-separated} if the sets $Kt$ for $t\in E$ are pairwise disjoint.
The following lemma is an analogue of \cite[Lemma 8.2]{KL13}.

\begin{lemma} \label{L-double}
Let $\cA\subseteq \sA$ be finite with positive independence density. Let $K\in \cF(\Gamma)$. Then there is some finite $\cA'\subseteq \sA$ with positive independence density such that each element of $\cA'$ is of the form ${\mathbf A}^E$ for some ${\mathbf A}\in \cA$ and some $K$-separated $E\in \cF(\Gamma\setminus K)$ with $|E|=2$.
\end{lemma}
\begin{proof} Denote by $q$ the independence density of $\cA$. Take a $K$-separated $E\in \cF(\Gamma\setminus K)$ with $q|E|\ge 2$.

Let $F\in \cF(\Gamma)$. Take a maximal $E$-separated subset $F'$ of $F$.
Then $E^{-1}EF'\supseteq F$, and hence
$$|F'|\ge |F|/|E|^2.$$

Note that $|EF'|=|E|\cdot |F'|$.
By assumption we can find a $J\subseteq EF'$ with $|J|\ge q|EF'|$ and some ${\mathbf A}\in \cA$ such that $J$ is an independence set for ${\mathbf A}$. For each $t\in E$, set $J_t=t^{-1}(J\cap tF')\subseteq F'$. Since $J\cap tF'$ for $t\in E$ is a partition of $J$, we have
$$\sum_{t\in E}|J_t|=\sum_{t\in E}|J\cap tF'|=|J|\ge q|EF'|=q|E|\cdot |F'|\ge 2|F'|.$$
Denote by $\eta$ the maximum of $|J_s\cap J_t|/|F'|$ for $s, t$ ranging over distinct elements of $E$. Then for each $t\in E$ there is some $W_t\subseteq J_t$ with $|W_t|\le \eta |F'|\cdot|E|$ such that the sets $J_t\setminus W_t$ for $t\in E$ are pairwise disjoint. Thus
\begin{align*}
2|F'|\le \sum_{t\in E}|J_t|=\sum_{t\in E}|W_t|+\big|\bigcup_{t\in E} (J_t\setminus W_t)\big|\le \eta |F'|\cdot |E|^2+|F'|,
\end{align*}
and hence $\eta\ge 1/|E|^2$. Then we can find distinct $s, t\in E$ with
$$|J_s\cap J_t|= \eta |F'|\ge |F'|/|E|^2\ge |F|/|E|^4.$$
Note that $t(J_s\cap J_t)\cup s(J_s\cap J_t)\subseteq J$, and $t(J_s\cap J_t)\cap s(J_s\cap J_t)\subseteq tF'\cap sF'=\emptyset$.
Thus $J_s\cap J_t$ is an independence set for ${\mathbf A}^{\{s, t\}}$. Therefore the set $\cA'$ consisting of ${\mathbf A}^{\{s, t\}}$ for ${\mathbf A}\in \cA$ and distinct $s, t\in E$ has independence density at least $1/|E|^4$.
\end{proof}

From Lemma~\ref{L-double} via induction on $n$ we have

\begin{lemma} \label{L-double double}
Let $\cA\subseteq \sA$ be  finite  with positive independence density. Let $K\in \cF(\Gamma)$ and $n\in \Nb$. Then there is some finite $\cA'\subseteq \sA$ with positive independence density such that each element of $\cA'$ is of the form ${\mathbf A}^E$ for some ${\mathbf A}\in \cA$ and some $K$-separated $E\in \cF(\Gamma\setminus K)$ with $|E|=2^n$.
\end{lemma}

Fix a compatible metric $\rho$  on $X$. For ${\mathbf A}=(A_1, \dots, A_k)\in \sA_k$, we set
$$\diam({\mathbf A}, \rho)=\max_{1\le j\le k}\diam(A_j, \rho).$$
For finite $\cA\subseteq \sA$, we set
$$\diam(\cA, \rho)=\max_{{\mathbf A}\in \cA}\diam({\mathbf A}, \rho).$$
For any $\varepsilon>0$, clearly every finite closed  $\cA\subseteq \sA$ has a closed splitting with diameter at most $\varepsilon$.

 For ${\mathbf A}\in \sA_k$, we set $|{\mathbf A}|=k$. For $s\in \Gamma$ and ${\mathbf A}=(A_1, \dots, A_k)\in \sA_k$, we set $s{\mathbf A}=(sA_1, \dots, sA_k)\in \sA_k$. The following is an analogue of \cite[Theorem 3.18]{KL07} and \cite[Theorem 8.1]{KL13}.

\begin{theorem} \label{T-chaos}
Let $\cA\subseteq \sA_{\ge 2}$ be  finite pairwise disjoint closed with positive independence density.
Then there are some ${\mathbf A}\in \cA$ and  a Cantor set $Z$  contained in the union of the entries of ${\mathbf A}$ such that for any finite set $Y\subseteq Z$ and any map $f: Y\rightarrow Z$ one has
$$ \liminf_{\Gamma\ni s\to \infty} \max_{y\in Y}\rho(sy, f(y))=0.$$
\end{theorem}
\begin{proof}
Take an increasing sequence $\{e_\Gamma\}\subseteq K_1\subseteq K_2\subseteq \cdots$ of finite subsets of $\Gamma$ with union $\Gamma$. We shall construct, via induction on $m$, finite  $\cA_m\subseteq \sA$ with the following properties:
\begin{enumerate}
\item $\cA_1\subseteq \cA$,
\item for every $m\ge 2$, there are maps $\pi_m: \cA_m\rightarrow \cA_{m-1}$ and $\zeta_m: \cA_m\rightarrow \Gamma$ such that for every ${\mathbf A}\in \cA_m$ one has $|{\mathbf A}|=2|\pi_m({\mathbf A})|$ and each entry of $\pi_m({\mathbf A})$ contains exactly two entries of $\zeta_m({\mathbf A}){\mathbf A}$,
\item when $m\ge 2$, for every ${\mathbf A}\in \cA_m$ defining $\xi_m({\mathbf A})\in \Gamma$ by $s_j=\zeta_j\pi_{j+1}\pi_{j+2}\cdots\pi_m({\mathbf A})$ for all $2\le j\le m$ and $\xi_m({\mathbf A})=s_2\cdots s_m$, we have $\diam(\xi_m({\mathbf A}){\mathbf A}, \rho)\le 2^{-m}$,
\item when $m\ge 2$, for every ${\mathbf A}\in \cA_m$, writing $\pi_m({\mathbf A})=(B_1, \dots, B_\ell)$ and ${\mathbf A}=(A_1, \dots, A_{2\ell})$, for any map $\gamma: [2\ell]\rightarrow [\ell]$, there is some
    $$u\in \Gamma\setminus \xi_{m-1}(\cA_{m-1})^{-1}K_{m-1}\xi_{m-1}(\cA_{m-1})\zeta_m({\mathbf A})$$
     such that $uA_j\subseteq B_{\gamma(j)}$ for all $j\in [2\ell]$, where $\xi_1(\cA_1)=\{e_\Gamma\}$,
\item for every $m$, $\cA_m$ is pairwise disjoint and closed,
\item for every $m$, $\cA_m$ has positive independence density.
\end{enumerate}

Suppose that we have constructed such $\cA_m$ over all $m$. Removing the elements of $\cA_m$ with some empty entry, we may assume that the entries of the elements of each $\cA_m$ are all nonempty.
Since each $\cA_m$ is nonempty and finite, the inverse limit space $\varprojlim_{m\to \infty} \cA_m$ for the maps $\pi_m$ is nonempty. Thus we can find ${\mathbf A}_m\in \cA_m$ for each $m\in \Nb$ such that $\pi_{m+1}({\mathbf A}_{m+1})={\mathbf A}_m$ for all $m$. For any $m\ge 2$, set
 ${\mathbf A}'_m=\xi_m({\mathbf A}_m){\mathbf A}_m$. Then for each $m\ge 2$, ${\mathbf A}_m'\in \sA_{|{\mathbf A}_1|2^{m-1}}$  and each entry of ${\mathbf A}_{m}'$ contains exactly two entries of ${\mathbf A}_{m+1}'$ by (2), and  $\diam({\mathbf A}_m', \rho)\le 2^{-m}$ by (3). Denote by $Z_m$ the union of the entries of ${\mathbf A}_m'$, and set $Z=\bigcap_{m\ge 2}Z_m$. Then $Z$ is a Cantor set. Since $\xi_2({\mathbf A}_2)=\zeta_2({\mathbf A}_2)$, by (2) the entries of ${\mathbf A}_2'=\xi_2({\mathbf A}_2){\mathbf A}_2=\zeta_2({\mathbf A}_2){\mathbf A}_2$ are contained in the entries of $\pi_2({\mathbf A}_2)={\mathbf A}_1$. Thus $Z\subseteq Z_2$ is contained in the union of the entries of ${\mathbf A}_1$.

 Let $Y\subseteq Z$ be finite, and let $f$ be a map $Y\rightarrow Z$. Let $K\in \cF(\Gamma)$ and $\varepsilon>0$.  Take $m\ge 2$ such that distinct elements of $Y$ lie in distinct entries of ${\mathbf A}'_{m+1}$, $K\subseteq K_m$, and $2^{-m}<\varepsilon$.  Write ${\mathbf A}_m'=(B_1, \dots, B_\ell)$ and ${\mathbf A}_{m+1}'=(A_1, \dots, A_{2\ell})$. Then there is some map $\gamma: [2\ell]\rightarrow [\ell]$ such that for any $y\in Y$, if $y\in A_j$ then $f(y)\in B_{\gamma(j)}$. Set $t_m=\xi_m({\mathbf A}_m)$ and $t_{m+1}=\xi_{m+1}({\mathbf A}_{m+1})=t_m\zeta_{m+1}({\mathbf A}_{m+1})$. Then ${\mathbf A}_m=(t_m^{-1}B_1, \dots, t_m^{-1}B_\ell)$ and
 ${\mathbf A}_{m+1}=(t_{m+1}^{-1}A_1, \dots, t_{m+1}^{-1}A_{2\ell})$. By (4) there is some $u\in \Gamma\setminus \xi_{m}(\cA_{m})^{-1}K_{m}\xi_{m}(\cA_{m})\zeta_{m+1}({\mathbf A}_{m+1})$ such that $ut_{m+1}^{-1}A_j\subseteq t_m^{-1}B_{\gamma(j)}$ for all $j\in [2\ell]$. For every $y\in Y$, say $y\in A_j$ for some $j\in [2\ell]$, one has $t_mut_{m+1}^{-1}y, f(y)\in B_{\gamma(j)}$ and hence
 $$ \rho(t_mut_{m+1}^{-1}y, f(y))\le \diam({\mathbf A}'_m, \rho)\le 2^{-m}<\varepsilon.$$
 Since $t_mut_{m+1}^{-1}\not\in K_m$, we have $t_mut_{m+1}^{-1}\not\in K$. Therefore
 $$\liminf_{\Gamma\ni s\to \infty} \max_{y\in Y}\rho(sy, f(y))=0.$$

We now construct the $\cA_m$. We set $\cA_1=\cA$. By assumption (5) and (6) are satisfied for $m=1$.
Assume that we have constructed $\cA_m$ with the above properties. Take $n\in \Nb$ such that $2^n\ge 2+|{\mathbf A}|^{2|{\mathbf A}|}$ for all ${\mathbf A}\in \cA_m$. By Lemma~\ref{L-double double} we can find a finite  $\cA_m'\subseteq \sA$ with positive independence density such that each element of $\cA_m'$ is of the form ${\mathbf A}^E$ for some ${\mathbf A}\in \cA_m$ and some $\xi_m(\cA_m)^{-1}K_m\xi_m(\cA_m)$-separated $E\in \cF(\Gamma)$ with $|E|=2^n$. Let ${\mathbf A}'\in \cA_m'$ and write it as ${\mathbf A}^E$ as above. Write ${\mathbf A}$ as $(A_1, \dots, A_\ell)$.
 Fix distinct $s_0, s_1\in E$, and take an injection $\varphi: [\ell]^{[2\ell]}\rightarrow E\setminus \{s_0, s_1\}$.
 For all $1\le i\le \ell$ and $1\le j\le 2$,
 take $\omega_{i, j}: E\rightarrow [\ell]$ such that $\omega_{i, j}(s_0)=i$, $\omega(s_1)=j$ and $\omega_{i, j}(\varphi(\gamma))=\gamma(i+(j-1)\ell)$ for all $\gamma: [2\ell]\rightarrow [\ell]$, and set
 $$ A_{i, j}=\bigcap_{s\in E}s^{-1}A_{\omega_{i, j}(s)}.$$
 Then ${\mathbf A}'':=(A_{1, 1}, \dots, A_{\ell, 1}, A_{1, 2}, \dots, A_{\ell, 2})\in \sA_{2\ell}$ is pairwise disjoint and closed, and
 every independence set for ${\mathbf A}'={\mathbf A}^E$ is an independence set for ${\mathbf A}''$. The family $\cA_{m+1}:=\{{\mathbf A}'': {\mathbf A}'\in \cA_m'\}$ clearly satisfies the conditions (5) and (6). Setting $\pi_{m+1}({\mathbf A}'')={\mathbf A}$ and $\zeta_{m+1}({\mathbf A}'')=s_0$, the property (2) is verified. For any map $\gamma: [2\ell]\rightarrow [\ell]$, we have
 $\varphi(\gamma)A_{i, j}\subseteq A_{\gamma(i+(j-1)\ell)}$ for all $1\le i\le \ell$ and $1\le j\le 2$. Since $E$ is $\xi_m(\cA_m)^{-1}K_m\xi_m(\cA_m)$-separated and $s_0\neq \varphi(\gamma)$, we have
 $$\varphi(\gamma)\not\in \xi_m(\cA_m)^{-1}K_m\xi_m(\cA_m)s_0=\xi_m(\cA_m)^{-1}K_m\xi_m(\cA_m)\zeta_{m+1}({\mathbf A}'').$$
 Thus the property (4) also holds.
 Replacing each ${\mathbf A}''$ by a suitable closed splitting of $\{{\mathbf A}''\}$, we also make (3) hold. This finishes the induction step.
 \end{proof}

Now Theorem~\ref{T-Li Yorke} follows from Theorems~\ref{T-positive naive} and \ref{T-chaos}.

\section{Positive Independence Density and Tameness} \label{S-tame}

It was shown in \cite[Theorem 7.1]{KL13} that if ${\mathbf A}\in \sA$ has positive independence density then ${\mathbf A}$ has an infinite independence set.
With a minor modification, the proof also works for finite families in $\sA$:

\begin{theorem} \label{T-IE to tame}
Let $\cA\subseteq \sA$ be finite  with positive independence density. Then at least one element of $\cA$ has an infinite independence set.
\end{theorem}

The action $\Gamma\curvearrowright X$ is untame exactly when there is a pairwise disjoint closed ${\mathbf A}\in \sA_{\ge 2}$ with an infinite independence set \cite[Proposition 8.14]{KL16}. Then Theorem~\ref{T-naive to tame} follows from Theorems~\ref{T-positive naive} and \ref{T-IE to tame}.

\section{An Action with Positive Naive Entropy but no non-diagonal Orbit IE-pairs} \label{S-example}

For $k\in \Nb$, recall that $(x_1, \dots, x_k)\in X^k$ is called an {\it orbit IE-tuple} (or {\it orbit IE-pair} when $k=2$) if for any product neighborhood $U_1\times \cdots \times U_k$ of $(x_1, \dots, x_k)$ in $X^k$, the tuple $(U_1, \dots, U_k)$ has positive independence density \cite[Definition 3.2]{KL13}. When $\Gamma$ is amenable, this is the same as IE-tuples defined in \cite[Definition 3.1]{KL07}.

When $\Gamma$ is amenable, $\Gamma\curvearrowright X$ has positive entropy exactly when $X$ has non-diagonal IE-pairs \cite[Proposition 3.9]{KL07} \cite[Theorem 12.19]{KL16}. When $\Gamma$ is sofic and $\Sigma$ is a sofic approximation sequence for $\Gamma$, $\Gamma\curvearrowright X$ has positive sofic entropy with respect to $\Sigma$ exactly when $X$ has non-diagonal $\Sigma$-IE-pairs \cite[Proposition 4.16]{KL13} \cite[Theorem 12.39]{KL16}. For general $\Gamma$, if $X$ has  non-diagonal orbit IE-pairs, then from Theorem~\ref{T-positive naive} we know that $\Gamma\curvearrowright X$ has positive naive entropy. We shall show in Proposition~\ref{P-positive naive no orbit IE} that the converse fails.

Denote by $\Zb\Gamma$ the integral group ring of $\Gamma$ \cite[page 3]{Passman} \cite[Section 13.1]{KL16}.
It consists of all functions $f: \Gamma\rightarrow \Zb$ with finite support. Writing $f$ as $\sum_{s\in \Gamma}f_s s$, the addition and multiplication of $\Zb\Gamma$ are defined by
\begin{align} \label{E-group ring}
\sum_s f_s s+\sum_s g_s s=\sum_s (f_s+g_s)s, \quad (\sum_sf_s s)(\sum_tg_t t)=\sum_t(\sum_sf_sg_{s^{-1}t})t.
\end{align}
It also has an involution $*$ defined by
$$ (\sum_s f_s s)^*=\sum_s f_{s^{-1}} s.$$

For any countable left $\Zb\Gamma$-module $\cM$, its Pontryagin dual $\widehat{\cM}$ consisting of all group homomorphisms $\cM\rightarrow \Rb/\Zb$ under pointwise multiplication and convergence is a compact metrizable abelian group, and $\Gamma$ acts on $\widehat{\cM}$ naturally by continuous automorphisms with $(s\varphi)(x)=\varphi(s^{-1}x)$ for all $\varphi\in \widehat{\cM}$, $s\in \Gamma$ and $x\in \cM$.
We refer the reader to \cite{KL16, LS, Schmidt} for general information on the study of $\Gamma\curvearrowright \widehat{\cM}$.

When $\cM=\Zb\Gamma$, we may identify $\widehat{\cM}$ with $(\Rb/\Zb)^\Gamma$, and the induced $\Gamma$-action on
$(\Rb/\Zb)^\Gamma$ is the left shift action given by $(sx)_t=x_{s^{-1}t}$ for all $x\in (\Rb/\Zb)^\Gamma$ and $s, t\in \Gamma$.

For any submodule $\cM'$ of $\cM$, the restriction map yields a factor map (i.e. a continuous surjective $\Gamma$-equivariant map) $\widehat{\cM}\rightarrow \widehat{\cM'}$. For $f\in \Zb\Gamma$, we have the $\Zb\Gamma$-module $\Zb\Gamma/\Zb\Gamma f$ and denote $\widehat{\Zb\Gamma/\Zb\Gamma f}$ by $X_{f}$. One may identify $X_{f}$ with the closed $\Gamma$-invariant subgroup of $(\Rb/\Zb)^\Gamma$ consisting of $x\in (\Rb/\Zb)^\Gamma$ satisfying $xf^*=0$ \cite[page 311]{Li12}, where the convolution product $xf^*$ is defined similar to \eqref{E-group ring}.

\begin{lemma} \label{L-no orbit IE}
Let $a\in \Gamma$ with infinite order. Then $X_{a-1}$ has no non-diagonal orbit IE-pairs.
\end{lemma}
\begin{proof} Note that $X_{a-1}=\{x\in (\Rb/\Zb)^\Gamma: x(a-1)^*=0\}$ consists of exactly those $x\in (\Rb/\Zb)^\Gamma$ satisfying $x_{ta}=x_t$ for all $t\in \Gamma$. Let $x$ and $y$ be distinct points in $X_{a-1}$. Then $x_s\neq y_s$ for some $s\in \Gamma$. Take open neighborhoods $V_x$ and $V_y$ of $x_s$ and $y_s$ in $\Rb/\Zb$ respectively such that $V_x\cap V_y=\emptyset$. Denote by $U_x$ ($U_y$ resp.) the set of $z\in X_{a-1}$ with $z_s\in V_x$ ($z_s\in V_y$ resp.). Then $U_x$ and $U_y$ are neighborhoods of $x$ and $y$ in $X_{a-1}$ respectively. For any distinct $k, m\in \Nb$, if $sa^kz\in U_x$ and $sa^mz\in U_y$ for some $z\in X_{a-1}$, then $z_{e_\Gamma}=z_{a^{-k}}=(sa^kz)_s\in V_x$ and $z_{e_\Gamma}=z_{a^{-m}}=(sa^mz)_s\in V_y$, which is impossible. Thus for any $n\in \Nb$, if $J$ is an independence set for $(U_x, U_y)$ contained in $\{sa^k: k=1, \dots, n\}$, then $|J|\le 1$. Therefore $(U_x, U_y)$ has independence density $0$, whence $(x, y)$ is not an orbit IE-pair of $X_{a-1}$.
\end{proof}

Now let $\Fb_2$ be the rank $2$ free group with generators $a$ and $b$.

\begin{proposition} \label{P-positive naive no orbit IE}
There is an action of $\Fb_2$ on a compact metrizable abelian group $X$ by continuous automorphisms such that $\Fb_2\curvearrowright X$ has positive naive entropy while $X$ has no non-diagonal orbit IE-pairs. Furthermore, there is a finite pairwise disjoint closed $\cA\subseteq \sA_2$ with positive independence density such that no element of $\cA$ has positive independence density.
\end{proposition}
\begin{proof} We shall show that $\Fb_2\curvearrowright X_{a-1}$ satisfies the conditions. From Lemma~\ref{L-no orbit IE} we know that $ X_{a-1}$ has no non-diagonal orbit IE-pairs.

Note that $\Zb\Fb_2$ has a free left $\Zb\Fb_2$-submodule with generators $a-1$ and $b-1$ \cite[Corollary 10.3.7.(iv)]{Passman}, and hence $\Zb\Fb_2/\Zb\Fb_2 (a-1)$ contains a $\Zb\Fb_2$-submodule isomorphic to $\Zb\Fb_2$.
Therefore the action $\Fb_2\curvearrowright X_{a-1}$ has a factor $\Fb_2\curvearrowright \widehat{\Zb\Fb_2}=(\Rb/\Zb)^{\Fb_2}$. As naive entropy does not increase under taking factors, we conclude that $\Fb_2\curvearrowright X_{a-1}$ has positive naive entropy.

To prove the last assertion of the proposition, assume conversely that every finite pairwise disjoint closed $\cA\subseteq \sA_2$ with positive independence density has an element with positive independence density. Take a compatible metric $\rho$ on $X_{a-1}$.
Since $\Fb_2\curvearrowright X_{a-1}$ has positive naive entropy, by Theorem~\ref{T-positive naive} there is some finite pairwise disjoint closed $\cA_1\subseteq \sA_2$ with positive independence density. By the assumption there is some $\mathbf{A}_1\in \cA_1$ with positive independence density. Inductively, assume that we have found some  closed $\mathbf{A}_k\in \sA_2$ with positive independence density. Take a finite closed splitting $\cA_{k+1}\subseteq \sA_2$ of $\{\mathbf{A}_k\}$ such that $\diam(\cA_{k+1}, \rho)\le \diam(X_{a-1}, \rho)/2^k$. By Proposition~\ref{P-splitting} we know that $\cA_{k+1}$ has positive independence density. Then by assumption we can find some $\mathbf{A}_{k+1}\in \cA_{k+1}$ with positive independence density. In this way we obtain a sequence $\{\mathbf{A}_k\}_{k\in \Nb}$ of closed elements in $\sA_2$ such that each $\mathbf{A}_k$ has positive independence density and $\diam(\mathbf{A}_k, \rho)\to 0$ as $k\to \infty$. Writing $\mathbf{A}_k=(A_{k, 1}, A_{k, 2})$, we may assume that $A_{k+1, i}\subseteq A_{k, i}$ for all $k\in \Nb$ and $i=1, 2$. Then for each $i=1, 2$, the intersection $\bigcap_{k\in \Nb}A_{k, i}$ is a singleton $\{x_i\}$. As $A_{1, 1}\cap A_{1, 2}=\emptyset$, we have $x_1\neq x_2$. Then $(x_1, x_2)$ is a non-diagonal orbit IE-pair, which is a contradiction. This proves the last assertion of the proposition.
\end{proof}

When $\Gamma$ is amenable, the independence density for each ${\mathbf A}\in \sA$ is a limit \cite[page 287]{KL16} and hence every finite $\cA\subseteq \sA$ with positive independence density has an element with positive independence density. Proposition~\ref{P-positive naive no orbit IE} shows that this fails for $\Fb_2$.

From \cite[Propositions 4.6 and 4.16]{KL13} we know that when $\Gamma$ is sofic, if $\Gamma\curvearrowright X$ has positive sofic entropy with respect to some sofic approximation sequence of $\Gamma$, then $X$ has a non-diagonal orbit IE-pair. Now let $\Fb_2\curvearrowright X$ be an action in Proposition~\ref{P-positive naive no orbit IE}. As $\Fb_2$ is sofic, $X$ has no non-diagonal orbit IE-pairs, and $\Fb_2\curvearrowright X$ has a fixed point, we conclude that $\Fb_2\curvearrowright X$ has sofic entropy zero with respect to every sofic approximation sequence of $\Fb_2$.
Thus results of \cite{KL13} do not tell us that $\Fb_2\curvearrowright X$ is Li-Yorke chaotic or untame.
On the other hand,
since $\Fb_2\curvearrowright X$ has positive naive entropy,  Theorems~\ref{T-Li Yorke} and \ref{T-naive to tame} imply that $\Fb_2\curvearrowright X$ is Li-Yorke chaotic and untame.


\end{document}